\documentclass[12pt]{amsart}

\usepackage{amsmath,amssymb,amsfonts,amsthm,amscd,amstext,amsxtra,amsopn,array,url,verbatim,mathrsfs,enumerate,anysize,soul}
\usepackage{graphicx}
\usepackage{amsmath,amssymb,amsfonts,amsthm,amssymb,amscd,url,amstext,amsxtra,amsopn
,txfonts}
\usepackage{verbatim}
\marginsize{2.25cm}{2.25cm}{2.25cm}{2.25cm}

\usepackage[dvipsnames,usenames]{xcolor}
\usepackage[colorlinks=true,urlcolor=Black,citecolor=Black,linkcolor=Black]{hyperref}
\usepackage{times}
\usepackage{mathabx}

\newtheorem{thrm}{Theorem}[section]
\newtheorem{lemma}[thrm]{Lemma}
\newtheorem{prop}[thrm]{Proposition}

\theoremstyle{definition}

\theoremstyle{remark}

\numberwithin{equation}{section}

\numberwithin{equation}{section}

\newcommand{\QQ}{\mathbb{Q}}
\newcommand{\CC}{\mathbb{C}}

\newcommand{\ZZ}{\mathbb{Z}}

\newcommand{\norm}{{\mathrm{N}}}

\newcommand{\ringO}{{\mathfrak{O}}}
\newcommand{\fa}{{\mathfrak{a}}}
\newcommand{\fb}{{\mathfrak{b}}}

\newcommand{\fp}{{\mathfrak{p}}}
\newcommand{\fm}{{\mathfrak{m}}}

\makeatletter
\def\imod#1{\allowbreak\mkern7mu({\operator@font mod}\,\,#1)}
\makeatother

\usepackage{mathtools}

\begin{document}

\title{Two dimensional value-distribution of cubic Hecke $L$-functions}

\thanks{Research of both authors is partially supported by NSERC}




\author{Amir Akbary}
\author{Alia Hamieh}

\address{Department of Mathematics and Computer Science \\
        University of Lethbridge \\
        Lethbridge, AB T1K 3M4 \\
        Canada}
        
        \address{Department of Mathematics and Statistics \\
        University of Northern British Columbia \\
        Prince George, BC V2N4Z9 \\
        Canada}

\email{amir.akbary@uleth.ca}
\email{alia.hamieh@unbc.ca}

\begin{abstract}

We establish the two-dimensional asymptotic distributions of the logarithm and logarithmic derivative of $L$-functions associated with a family of cubic Hecke characters. A crucial ingredient in the proof of our main result is an exponential decay estimate for the characteristic functions of the distributions.
\end{abstract}

\maketitle
\section{INTRODUCTION}

A classical result of Bohr and Jessen \cite{BJ} states that for fixed $\sigma>1/2$ and varying $t$, the values $\log{\zeta(\sigma+it)}$ of the Riemann zeta function $\zeta(\sigma+it)$ have a limiting distribution with a continuous density. The original proof of this theorem uses the properties of the sums of convex curves. In an important paper \cite{JW} Jessen and Wintner described a general framework for Bohr-Jessen's theorem. Their method uses ideas from probability theory such as sums of independent random variables and infinite convolutions together with Fourier analysis machinery.  Using this approach in  \cite[Theorem 19]{JW} they provided detailed information on the distribution function in Bohr-Jessen's theorem; for example, among other things they proved that the density function is continuous and possesses continuous partial derivatives of all orders. Hattori and Matsumoto \cite{HM} extended this line of research by studying the tail of the distribution of $\log\zeta(\sigma+it)$ for $\frac12<\sigma<1$. Their results were strengthened by Lamzouri \cite{lamzouri2} and Lamzouri, Lester, and Radziwi\l\l \; \cite{LLR} in which the authors also investigate the discrepancy between the distributions of $\log\zeta(\sigma+it)$ and that of an adequately chosen random variable $\log\zeta(\sigma,X)$. 

Over the past few decades similar probabilistic approaches have been used in studying the value distribution of other families of $L$-functions. A notable case is the family of quadratic twists in which one studies the values of certain real functions attached to quadratic characters $\chi_d$ (such as $\log{|L(s, \chi_d)|}$, ${\rm arg} {L(s, \chi_d)}$, or $L(1, \chi_d))$, as $d$ varies over the fundamental discriminants (see, for example, \cite{chowla-erdos}, \cite{elliott1}, \cite{elliott2}, \cite{GS}, 
and \cite{lamzouri2}).  Here $L(s, \chi)$ denotes the Dirichlet $L$-function associated with the Dirichlet character $\chi$. We note that in spite of the vast literature on the one-dimensional distributions of such families, the two dimensional distributions for discrete families of $L$-functions are not widely studied. An example of a two-dimensional distribution theorem for the family $L(s, \chi_d)$ is proved by Stankus \cite{S}.

More recently Ihara and Matsumoto initiated a systematic study, in spirit of Jessen-Wintner theory, of two-dimensional value-distributions of logarithms and logarithmic derivatives of families of $L$-functions. The following is proved in \cite[Theorem 1.1 and Proposition 3.5]{IM10}.
\begin{thrm}[Ihara-Matsumoto]
\label{IM-first}
Let $\hat{H}_f$ be the set of all primitive characters modulo $f$ and $\pi(m)$ be the number of primes not exceeding $m$. Let $\sigma:=\Re(s) >1/2$ be fixed and let $|dw|={(dx dy)}/{2\pi}$.  Then there exists a continuous non-negative density function $M_\sigma(w)$ such that 
\begin{equation*}
\lim_{m\rightarrow \infty}\frac{1}{\pi(m)}\sum_{\substack{{2<f\leq m}\\{{f}~\text{prime}}}}
\frac{1}{f-2}\#\{\chi_{f} \in {\hat{H}_{{f}}};~
\log L(s, \chi_{f})\in S\}=\int_{S}M_{\sigma}(w)\;|dw|, 
\end{equation*}
where $S\subset \CC$ is either compact or complement of a compact set. Moreover, $M_\sigma(\bar{w})=M_\sigma(w)$ and $M_\sigma (w)$ tends to zero as $|w|\rightarrow \infty$.
\end{thrm}

In \cite[Section 3]{IM10} the density function $M_\sigma(w)$ is explicitly constructed as an infinite convolution of the local density functions given in \cite[formula (24)]{IM10}. In addition it is proved that the Fourier transform of $M_\sigma(w)$ is continuous in $\sigma$ and $w$, and for each $\sigma>1/2$ it belongs to 
$L^t$ for any $1\leq t \leq +\infty$ (see \cite[Proposition 3.4]{IM10}). 

Jessen-Wintner approach can also be applied to the study of the two-dimensional value distribution of families of Hecke $L$-functions on number fields or on function fields. For example,
let $k$ be $\QQ$ or an imaginary quadratic number field. For an integral ideal $\mathfrak{f}$, denote by $H_\mathfrak{f}$ the ray class group of $k$ of conductor $\mathfrak{f}$. Let $I_\mathfrak{f}$ be the group of fractional ideals in $k$ relatively prime to $\mathfrak{f}$, and let $i_\mathfrak{f}: I_{\mathfrak{f}} \rightarrow H_\mathfrak{f}$ be the projection map. For $\chi_\mathfrak{f} \in {\hat{H}}_\mathfrak{f}$ (the collection of primitive characters of $H_\mathfrak{f}$) and an integral ideal $\mathfrak{a}$, set $\chi_{\mathfrak{f}}(\mathfrak{a})=\chi_{\mathfrak{f}}(i_{\mathfrak{f}}(\mathfrak{a}))$ if $(\mathfrak{a}, \mathfrak{f})=1$, and $\chi_{\mathfrak{f}}(\mathfrak{a})=0$ otherwise. Let $L(s, \chi_\mathfrak{f})$ be the Hecke $L$-function associated to $\chi_\mathfrak{f}$. Let $\mathcal{L}(s, \chi_\mathfrak{f})$ be either $\log{L(s, \chi_\mathfrak{f})}$ or ${L^\prime}/{L}(s, \chi_\frak{f})$.
The following is proved in \cite[Theorem 4]{I-M}.

\begin{thrm}[Ihara-Matsumoto]
\label{IM-theorem}
Let $\sigma:=\Re(s) >1/2$ be fixed and let $|dw|={(dx dy)}/{2\pi}$. Assume the generalized Riemann hypothesis (GRH) for
$L(s, \chi_\mathfrak{f})$. Let ${\rm N}(\mathfrak{f})$ denote the norm of the ideal $\mathfrak{f}$. Then there exists a density function $M_\sigma(w)$, satisfying properties described in Theorem \ref{IM-first},  such that 
\begin{equation*}
\lim_{\substack{{{\rm N}(\mathfrak{f})\to\infty}\\{\mathfrak{f}~\text{prime}}}}
\frac{1}{{\rm N}(\mathfrak{f})}\#\{\chi_\mathfrak{f} \in {\hat{H}_{\mathfrak{f}}};~
\mathcal{L} (s, \chi_\mathfrak{f})\in S\}=\int_{S}M_{\sigma}(w)\;|dw|, 
\end{equation*}
where $S\subset \CC$ is either compact or complement of a compact set.
\end{thrm}

In this paper we establish an unconditional theorem in spirit of the above result for a family of cubic twists. 
Let $k=\QQ(\sqrt{-3})$ and $\zeta_3=\exp\left(2\pi i/3\right)$. Then  $\ringO_k=\ZZ[\zeta_3]$ is the ring of integers of $k$. Let \[\mathcal{C}:=\left\{c\in\mathfrak{O}_k;~c\neq 1 \text{ is square free and } c\equiv 1 \imod{\langle9\rangle} \right\}.\] 
For $c\in \mathcal{C}$, let $\chi_c =\left(\frac{.}{c}\right)_3$ be the cubic residue character modulo $c$.
We set 
\[ \mathcal{L}(s,\chi_{c})=\begin{cases} 
      \log L(s,\chi_{c}) & \text{(Case 1),} \\
       {L^\prime}/{L}(s,\chi_{c}) & \text{(Case 2).} 
   \end{cases}
\]

In \cite[Theorem 1.4]{AH} we proved a one-dimensional distribution result for $\Re(\mathcal{L}(\sigma, \chi_c)$ for a fixed $\sigma>\frac{1}{2}$ as $c$ varies in $\mathcal{C}$. The goal of this paper is to determine the two-dimensional limiting distribution of the values $\mathcal{L}(s,\chi_c)=\Re\left(\mathcal{L}(s,\chi_{c})\right)+i \Im\left(\mathcal{L}(s,\chi_{c})\right)$ for a fixed $s$ with $\Re(s)>\frac12$ as $c$ varies in $\mathcal{C}$.  More precisely, we prove the following theorem. 
\begin{thrm}\label{mainthrm}

Let $s\in\mathbb{C}$ be such that $\Re(s)>\frac12$.  Let $\mathcal{N}(Y)$ be the number of elements $c\in \mathcal{C}$ with norm not exceeding $Y$.
Then the following statements hold:

\noindent (i)  There exists a smooth probability density function $M_{s}(t_1, t_2)$ such that
\begin{align*}&\lim_{Y\to\infty}\frac{1}{\mathcal{N}(Y)}\#\left\{{c}\in \mathcal{C}: \norm(c)\leq Y,\;\;  \Re\left(\mathcal{L}(s,\chi_{c})\right) \leq z_1,~  \text{and}\;\;\ \Im\left(\mathcal{L}(s,\chi_{c})\right) \leq z_2 \right\}\\&\hspace{1em}=\int_{-\infty}^{z_1} \int_{-\infty}^{z_2} M_{s}(t_1, t_2)\; dt_1 dt_2 .\end{align*}
 The function $M_{s}(t_1, t_2)$ and its partial derivatives tend to zero as $|t_1+it_2|\to\infty$. If $\frac12<\Re(s)<1$, then $M_{s}(t_1,t_2)$ is real analytic and moreover, in (Case 2), $M_s(t_1, t_2)$ is  real analytic for  $\Re(s)=1$ as well. In addition, it satisfies $M_{s}(t_1,-t_2)=M_{\bar{s}}(t_1,t_2)$.

\noindent (ii) The asymptotic distribution function $F_s(z)$
can be constructed as an infinite convolution over prime ideals $\fp$ of $k$,
\[F_{s}(z)=\mbox{*}_{\fp}\;F_{s,\fp}(z),\]
where
\[F_{s,\fp}(z)=\begin{cases}\frac{1}{\norm(\fp)+1}\delta(z)+\frac{1}{3}\left(\frac{\norm(\fp)}{\norm(\fp)+1}\right)\sum_{j=0}^{2}\delta_{-a_{\fp,j}}(z) & \text{if}\; \fp\nmid\langle3\rangle,\\ \delta_{-a_{\fp,0}}(z) & \text{if}\; \fp\mid \langle3\rangle.
\end{cases}\] 
Here $\delta_{a}(z):=\delta(z-a)$, $\delta$ is the Dirac distribution, and
\[a_{\fp,j}:=a_{\fp,j}(s) =\begin{cases} 
    \log(1-\zeta_{3}^{j}\norm(\fp)^{-s})   & \text{ in (Case 1),} \\
\frac{\zeta_{3}^{j}  \log\norm(\fp) }{\norm(\fp)^{s}-\zeta_{3}^{j}}  & \text{in (Case 2).} 
   \end{cases}
   \]
\noindent (iii) The density function $M_s(t_1, t_2)$ can be constructed as the inverse Fourier transform of the characteristic function 
$\varphi_{F_s}(y_1, y_2)$, which in (Case 1) is given (for $y=y_1+iy_2$) by
 \begin{align*}
\varphi_{F_s}(y_1, y_2)&= \exp\left(-i\Re\left(\bar{y}\log\left(1-3^{-s}\right)\right)\right)\\&\hspace{1em}\times\prod_{\fp\nmid\langle3\rangle}
\left( \frac{1}{\norm(\fp)+1}+\frac{1}{3}\frac{\norm(\fp)}{\norm(\fp)+1}\sum_{j=0}^{2}\exp\left(-i\Re\left(\bar{y}\log\left(1-\frac{\zeta_{3}^{j}}{\norm(\fp)^{s}}\right)\right)\right) \right)
\end{align*}
and in (Case 2) is given by 
\begin{align*}
\varphi_{F_s}(y_1, y_2)&= \exp\left(-i\Re\left(\frac{\bar{y}\log 3}{3^{s}-1}\right)\right)\\&\hspace{1em}\times\prod_{\fp\nmid\langle3\rangle} \left( \frac{1}{\norm(\fp)+1}+\frac{1}{3}\frac{\norm(\fp)}{\norm(\fp)+1}\sum_{j=0}^{2}\exp\left(-i\Re\left(\frac{\bar{y}\zeta_{3}^{j}   \log\norm(\fp)    }{\norm(\fp)^{s}-\zeta_{3}^{j}}\right)\right) \right).
\end{align*}

\end{thrm}

Although the proof of Theorem \ref{mainthrm} shares common features with the proof of Theorem \ref{IM-theorem}, it differs from it in several aspects. Firstly, we employ the zero-density estimates for $L$-functions to prove our theorem unconditionally, without the assumption of the GRH. Secondly, the proof of the existence of the distributions in Theorem \ref{IM-theorem} relies on the construction of the characteristic functions of the distributions via the infinite convolutions of certain explicitly given local density functions (see \cite[Section 4, Theorem $\tilde{M}$ (iv) and references there]{I-M}). In contrast,
we establish the existence of the distribution functions in Theorem \ref{mainthrm} by describing the characteristic functions of the distributions as some complex moments of the related $L$-functions and then construct the associated density functions by an appeal to the exponential decay of the characteristic functions (see Lemma \ref{mainlemma2}).

Here we describe the strategy of our proof of Theorem \ref{mainthrm}. 
We say that $f$ possesses an asymptotic distribution function $F$ if
$$\lim_{N\rightarrow \infty}\frac{\#\{n\leq N;~f_1(n) \leq z_1~\text{and} ~ f_2(n)\leq z_2\}}{N}= F(z_1, z_2),$$
for all $(z_1, z_2)$ in which $F((-\infty, z_1) \times (-\infty, z_2))=F((-\infty, z_1] \times (-\infty, z_2])$ 
(see Section \ref{sectiontwo} for more explanation). 
The following lemma outlines our approach in proving Theorem \ref{mainthrm}.

\begin{lemma}
\label{mainlemma2}
Let $f=f_1+if_2$ be a complex valued arithmetic function and let $y=y_1+iy_2$. Suppose that, as $N\rightarrow \infty$, the functions
$$\frac{\displaystyle{\sum_{n=1}^{\infty}} e^{i(y_1 f_1(n)+y_2 f_2(n))}e^{-n/N}}{\displaystyle{\sum_{n=1}^{\infty}}e^{-n/N}} $$
converge uniformly on any sphere $|y|\leq a$ to a function $\widetilde{M}(y_1, y_2)$. Then $f$ possesses an asymptotic distribution function $F$. In this case, $\widetilde{M}$ is the characteristic function  of $F$. Moreover, if 
\begin{equation}
\label{decay}
\left|\widetilde{M}(y_1, y_2)\right|\leq \exp\left(-\eta\left|y\right|^{\gamma}\right),
\end{equation}
for some $\eta,\gamma>0$,
then $F(z_1, z_2)=\int_{-\infty}^{z_1}  \int_{-\infty}^{z_2}M(t_1, t_2)\;dt_1 dt_2$ for a smooth function $M$, where
\begin{equation}
\label{inversion}
M(z_1, z_2)=(1/2\pi)^2\int_{\mathbb{R}^2}\exp\left(-i(z_1y_1+z_2y_2)\right)\widetilde{M}(y_1, y_2)\;dy_1 dy_2.
\end{equation}
In addition if \eqref{decay} holds for $\gamma \geq 1$, then $M$ is real analytic.
\end{lemma}

We shall establish the following two propositions so as to verify that the conditions of Lemma \ref{mainlemma2} are satisfied by the family $\mathcal{L}(s, \chi_c)=\Re(\mathcal{L}(s,\chi_{c}))+i \Im(\mathcal{L}(s,\chi_{c}))$ described in Theorem  \ref{mainthrm}.

\begin{prop}\label{mainprop1}
Let \[\mathcal{N}^{*}(Y)=\sum_{c\in\mathcal{C}}\exp(-\norm(c)/Y).\] Fix $s$ with $\sigma=\Re(s)>\frac12$ and $y=y_1+iy_2\in\mathbb{C}$. Let $\widetilde{M}_s(y_1, y_2)$ be the function given by one of the product formulas in Theorem \ref{mainthrm}. Then
\begin{equation*}
\lim_{Y\to\infty}\frac{1}{\mathcal{N}^{*}(Y)} \sum_{{{c}}\in \mathcal{C}}^{\star}\exp\left(i(y_1\Re(\mathcal{L}(s,\chi_{c}))+y_2 \Im(\mathcal{L}(s,\chi_{c})))\right)\exp(-\norm({c})/Y)=\widetilde{M_s}(y_1, y_2)
\end{equation*} 
uniformly on any sphere $|y|\leq a$,  where $\star$ indicates that the sum is over $c$ such that $L(s,\chi_{c})\neq 0$. 
\end{prop}

\begin{prop}\label{mainprop2}
Fix $s$ with $\sigma=\Re(s)>\frac12$. As $|y|\rightarrow\infty$ , we have \[\left|  \widetilde{M_s}(y_1, y_2)
   \right|\leq\begin{cases}\exp\left(-C_1|y|^{\frac1{\sigma}}\left(\log\left| y\right|\right)^{-1}\right)& \text{(Case 1)}\\ \exp\left(-C_2|y|^{\frac1{\sigma}}\left(\log\left| y\right|\right)^{\frac{1}{\sigma}-1}\right) &\text{(Case 2)},\end{cases}\] where $C_1$ and $C_2$ are positive constants that depend only on $\sigma$.
\end{prop}

In establishing value-distribution results 
a rapid decay estimate for the characteristic function of the distribution is almost always a crucial part of the proof regardless of the approach employed (see for example \cite[Section~6]{elliott2}, \cite[Formula (38)]{IM10},  \cite[Lemma~4]{MM},  \cite[Proposition~1.11]{AH}, and \cite[Section~6]{LLR}).  
A common approach in proving results similar to Proposition \ref{mainprop2} uses an interpretation of the characteristic function of the distribution as the characteristic function of an infinite sum of certain independent random variables. Then the desired decay estimate of the characteristic function can be explored by exploiting the information on the statistical parameters, such as mean and variance, of the associated random variables (see for example \cite[Proposition 2.1]{LL}).
In this paper we prove the stated decay estimate in Proposition \ref{mainprop2} by an elementary method in spirit of \cite[Proposition 1.11]{AH}, which itself is based on ideas from \cite{MM} and \cite{W}.
We further note that, in proving \cite[Proposition~1.11]{AH} the authors adapted the method used in \cite{MM} in a relatively straightforward manner to get the desired upper bound for the characteristic function of the 1-dimesional distribution of the family $\Re\left(\mathcal{L}(\sigma,\chi_c)\right)$ for a fixed $\sigma>\frac12$ as $c$ varies in $\mathcal{C}$. In fact, one can easily check that the characteristic function of the distribution given in \cite[Theorem 1.4]{AH} can be expressed as $\varphi_{F_{\sigma}}(2y_1,0)$, where $\varphi_{F_s}(y_1,y_2)$ is given in Theorem \ref{mainthrm} above. While the proof of the exponential decay of $\varphi_{F_s}(y_1,y_2)$ is inspired by similar ideas, the details are more subtle and intricate owing to the nature of the characteristic function $\varphi_{F_s}(y_1,y_2)$ when $t=\Im(s)$ and $y_2$ are non-zero.

In Section \ref{sectiontwo}, after reviewing preliminaries on distribution functions, we prove Lemma \ref{mainlemma2} and show how it together with Propositions \ref{mainprop1} and \ref{mainprop2} imply Theorem \ref{mainthrm}. The proof of Proposition \ref{mainprop1} is given in Section \ref{sec:prop1}, and it follows to a great extent the proof of \cite[Proposition 1.10]{AH}. The proof of Proposition \ref{mainprop2}, however, requires novel ideas and is given in full detail in Section \ref{sec:prop2}.

\section{Distributions and Proofs of Lemma \ref{mainlemma2} and Theorem \ref{mainthrm}}
\label{sectiontwo}

Let $\mathbb{R}^k$ be the $k$-dimensional Euclidean space equipped with the Lebesgue measure. Let $F$ be a probability measure defined on the Borel $\sigma$-algebra $\mathcal{B}$ of $\mathbb{R}^k$. The set function $F: \mathcal{B} \rightarrow [0, 1]$ sending $B$ to $F(B)$ is called a \emph{distribution function} in $\mathbb{R}^k$. A set $E\in \mathbb{B}$ is called a \emph{continuity set} of the distribution function $F$ if $$F ({\rm Int}(E))=F(\bar{E}).$$ 
Here ${\rm Int}(E)$ is the collection of the interior points of $E$ and $\bar{E}$ denotes the closure of $E$. One can show that there is a countable set $C:=C(F)\subset \mathbb{R}$ such that $[a_1, b_1]\times \cdots \times [a_k, b_k]$ is a continuity set of $F$ for all $a_i, b_i \not\in C$. We say that the distribution function $F$ is \emph{absolutely continuous} if there is a Lebesgue integrable function $M(t):=M(t_1, \cdots, t_k)$ defined on $\mathbb{R}^k$ such that $$F(E)=\int_{E} M(t)\; dt$$
for any Borel set $E\in \mathcal{B}$.

For $y=(y_1, \cdots, y_k)$ and $z=(z_1, \cdots, z_k)$, set $\langle y, z \rangle=y_1 z_1+\cdots+y_kz_k$. The \emph{characteristic function} $\varphi_F$ of the distribution function $F$ is defined as the Fourier transform of the measure $F$. More precisely, $$\varphi_F(y):=\int_{\mathbb{R}^k} e^{i\langle y, z \rangle} \;dF(z).$$

We say that a sequence $(F_n )_{n=1}^{\infty}$ of distribution functions \emph{converges weakly} to a distribution function $F$, and we write $F_n \rightarrow F$, if $F_n (E) \rightarrow F(E)$ for all continuity sets $E$ of $F$. We know that $F_n \rightarrow F$ if and only if
$$\int_{\mathbb{R}^k} f(z)\; dF_n(z) \rightarrow \int_{\mathbb{R}^k} f(z)\; dF(z)$$
for any bounded continuous function $f(z)$ on $\mathbb{R}^k$.

We can show that the correspondence between the class of all distribution functions and the class of their characteristic functions is a one to one correspondence (see \cite[p. 53]{JW}). Moreover, the following important theorem exhibits a close connection between the convergence of a sequence of distribution functions and the corresponding sequence of their characteristic functions. This theorem is essentially an elaborate version of Levy's continuity theorem for $k$-dimensional distributions as described by Jessen and Wintner in \cite[pp. 53 and 54]{JW}.

\begin{thrm}
\label{Levy}
(i) Let $(F_n)_{n=1}^{\infty}$ be a sequence of distribution functions, and let $(\varphi_{F_n})_{n=1}^{\infty}$ be the sequence of their characteristic functions. Then $F_n$ converges weakly to a distribution function $F$ if and only if $\varphi_{F_n}$ converges uniformly to $\varphi$ in every sphere $|y|\leq a$.
Furthermore, in this case, $\varphi$ is the characteristic function of $F$, i.e., $\varphi=\varphi_F$.

(ii) In part (i) 
 if $y^{a}\varphi(y)\in L^{1}$ for an integer $a\geq 0$, then $F$ is absolutely continuous and its density $M(t)$, given by the inversion formula
 \begin{equation}
 \label{inversion formula}
M(t)= (1/2\pi)^k \int_{\mathbb{R}^k} e^{-i\langle t, y\rangle} \varphi(y)\; dt,
\end{equation}
is continuous, approaches zero as $|t|\rightarrow \infty$, and in the case $a>0$ has continuous partial derivatives of order $\leq a$, which may be obtained by differentiation under integral sign of  \eqref{inversion formula}, that approach zero as $|t| \rightarrow \infty$.
 
(iii) In part (i) if $$\varphi(y)=O\left(\exp(-A|y|) \right),$$
as $|y| \rightarrow \infty$, for some $A>0$, then $F$ is absolutely continuous and its density $M(t)=M(t_1, \cdots, t_k)$ is real analytic. In other words there is a neighbourhood of  each
point $(t_1^0, \cdots, t_k^0)\in \mathbb{R}^k$  in which $M(t)$ can be developed as a power series in terms of $t_1-t_1^0, \cdots, t_k-t_k^0$.
\end{thrm}

We say that $f(n)= f_1(n)+if_2(n)$ possesses an asymptotic distribution function {$F$}, if
$$\lim_{N\rightarrow \infty}\frac{\#\{n\leq N;~f_1(n) \leq z_1~\text and ~ f_2(n)\leq z_2\}}{N}= F(z_1, z_2),$$
for all $(z_1, z_2)$ in which $E_{z_{1}, z_{2}}=(-\infty, z_1] \times (-\infty, z_2]$ is a continuity set of $F$.

Theorem \ref{Levy} is the key tool in the proof of Lemma \ref{mainlemma2} that provides a criteria for proving the asymptotic distribution of the real and imaginary parts of the values of the cubic twists.

\begin{proof}[Proof of Lemma \ref{mainlemma2}]
The proof is a direct consequence of Theorem \ref{Levy} and a Tauberian theorem of Hardy and Littlewood (see \cite[Proof of Lemma 1.9]{AH} for details).
\end{proof}

The \emph{convolution} of two distribution functions $F$ and $G$ is the distribution function $F*G$ defined by
$$(F*G)(z)=\int_{\mathbb{R}^k} F(z-y)\; dG(y)=\int_{\mathbb{R}^k} G(z-y)\; dF(y).$$
One can show that $\varphi_{F*G}=\varphi_F \varphi_G.$  A distribution function $F$ is called the infinite convolution of distribution functions $F_1, F_2, \ldots, F_n, \ldots$ if $F_1*F_2*\ldots *F_n$  converges weakly to $F$ as $n\rightarrow \infty$. In such case we write $F=*_i F_i$. The following theorem provides a necessary and sufficient condition for the existence of infinite convolutions.
\begin{thrm}
\label{infinite-convolution}
The infinite convolution $*_i F_i$ exists if and only if there exists $\delta>0$ such that for $|y|\leq \delta$ we have
$$\lim_{m,n\rightarrow \infty} \prod_{m<j\leq n} \varphi_{F_j}(y)=1.$$
\end{thrm}
\begin{proof}
See \cite[Theorem 1 and footnote on page 53]{JW}.
\end{proof}
 
We now have all the ingredients needed for the proof of our main theorem. 
 
\begin{proof}[Proof of Theorem \ref{mainthrm}] 
 The proof comes as an application of Lemma \ref{mainlemma2} together with Proposition \ref{mainprop1} and Proposition \ref{mainprop2}.
 In Proposition \ref{mainprop1} we show that
\begin{equation*}
\lim_{Y\to\infty}\frac{1}{\mathcal{N}^{*}(Y)} \sum_{{{c}}\in \mathcal{C}}^{\star}\exp\left(i(y_1\Re(\mathcal{L}(s,\chi_{c}))+y_2 \Im(\mathcal{L}(s,\chi_{c})))\right)\exp(-\norm({c})/Y)=\widetilde{M_s}(y_1, y_2)
\end{equation*} 
uniformly on any sphere $|y|\leq a$, where $\widetilde{M_s}(y_1, y_2)$  is first obtained as the following Dirichlet series (see Section \ref{sec:prop1} for derivation and definition of $\lambda_y$):

{\small\begin{align*}\widetilde{M}_{s}(y_1,y_2)&= \sum_{r_{1},r_{2}>0}\frac{\lambda_{\overline{y}}((1-\zeta_{3})^{r_1})\lambda_{y}((1-\zeta_{3})^{r_2})}{3^{r_1s+r_2\overline{s}}}
\\&\hspace{2em}\times\sum_{\substack{\fa,\fb,\fm\subset\mathfrak{O}_{k}\\\gcd(\fa\fb\fm,(3))=1\\\gcd(\fa,\fb)=1}}\frac{\lambda_{\overline{y}}(\fa^3\fm)\lambda_{y}(\fb^3\fm)}{\norm(\fa)^{3s}\norm(\fb)^{3\overline{s}}\norm(\fm)^{2\sigma}\displaystyle{\prod_{\substack{\fp|\fa\fb\fm\\\fp\; \text{prime}}}\left(1+\norm(\fp)^{-1}\right)}}.
\end{align*}}
Following the computations in  \cite[Section~4.6]{AH}, one can verify that $\widetilde{M}_{s}(y_1,y_2)$ admits the product expansion given by {\small \begin{align}\label{prod1}
\widetilde{M}_{s}(y_1,y_2)&= \exp\left(-i\Re\left(\bar{y}\log\left(1-3^{-s}\right)\right)\right)\nonumber\\&\hspace{1em}\times\prod_{\fp\nmid\langle3\rangle}
\left( \frac{1}{\norm(\fp)+1}+\frac{1}{3}\frac{\norm(\fp)}{\norm(\fp)+1}\sum_{j=0}^{2}\exp\left(-i\Re\left(\bar{y}\log\left(1-\frac{\zeta_{3}^{j}}{\norm(\fp)^{s}}\right)\right)\right) \right)
\end{align}}
 in (Case 1) and 
{\small\begin{align}\label{prod2}
\widetilde{M}_{s}(y_1,y_2)&= \exp\left(-i\Re\left(\frac{\bar{y}\log 3}{3^{s}-1}\right)\right)\nonumber\\&\hspace{1em}\times\prod_{\fp\nmid\langle3\rangle} \left( \frac{1}{\norm(\fp)+1}+\frac{1}{3}\frac{\norm(\fp)}{\norm(\fp)+1}\sum_{j=0}^{2}\exp\left(-i\Re\left(\frac{\bar{y}\zeta_{3}^{j}   \log\norm(\fp)    }{\norm(\fp)^{s}-\zeta_{3}^{j}}\right)\right) \right)
\end{align}}
in (Case 2). It follows from Lemma \ref{mainlemma2} that $\mathcal{L}(s,\chi_c)$ has an asymptotic distribution function $F_{s}$ with the characteristic function $\varphi_{F_s}= \widetilde{M}_s$. 
In view of  Proposition \ref{mainprop2}, we also see that Lemma \ref{mainlemma2} establishes the existence of a smooth probability density function $M_s(t_1, t_2)$ for which $F_s(z_1, z_2)=\int_{-\infty}^{z_1} \int_{-\infty}^{z_2} M_s(t_1, t_2) \;dt_1 dt_2$, and 
 \begin{equation*}
M_s(z_1, z_2)=(1/2\pi)^2\int_{\mathbb{R}^2}\exp\left(-i(z_1y_1+z_2y_2)\right)\widetilde{M}_s(y_1, y_2)\;dy_1 dy_2.
\end{equation*}
One can verify that $\overline{M_{s}(t_1,t_2)}=M_{s}(t_1,t_2)$, i.e.,  $M_s$ is real, using the fact that $\overline{\varphi_{F_s}(y_1,y_2)}=\varphi_{F_{s}}(-y_1,-y_2)$. 
That $M_{s}(t_1,-t_2)=M_{\bar{s}}(t_1,t_2)$ follows from 
the identity 
$\varphi_{F_s}(y_1,-y_2)=\varphi_{F_{\bar{s}}}(y_1,y_2)$. This establishes parts (i) and (iii).

In order to get part  (ii) in which $F_{s}$ is given as an infinite convolution, we apply Theorem \ref{infinite-convolution} in addition to the observation that 
\begin{equation*}
\varphi_{F_s}=\prod_{\fp}\varphi_{{F_s},\fp},
\end{equation*}
 where the local factors $\varphi_{{F_s},\fp}$  are determined by the product formulae  \eqref{prod1} and \eqref{prod2}. See \cite[Proof of Theorem 1.4]{AH} for details.
\end{proof}

\section{Proof of Proposition \ref{mainprop1}}\label{sec:prop1}

The arguments in the proof of \cite[Proposition 1.10]{AH} can be followed {\it mutatis mutandis} to establish Proposition \ref{mainprop1}. In this section, we provide an outline for the proof. We refer the reader to the corresponding parts in \cite{AH} for all the details while highlighting the differences and the necessary adjustments whenever needed.

In what follows we only consider (Case 2) since (Case 1) can be treated similarly. We start by making the following simple but crucial observation. For $y=y_1+iy_2$, we set 
\begin{align*}I_{y}(s,\chi_c)&=\exp\left(i(y_1\Re(\mathcal{L}(s,\chi_{c}))+y_2 \Im(\mathcal{L}(s,\chi_{c})))\right)\\&=\exp\left(\frac{i}{2}(\overline{y}\mathcal{L}(s,\chi_{c}))+y \overline{\mathcal{L}(s,\chi_{c})})\right)\\&=\exp\left(\frac{i}{2}(\overline{y}\mathcal{L}(s,\chi_{c}))\right)\exp\left(\frac{i}{2}(y \overline{\mathcal{L}(s,\chi_{c})})\right).\end{align*}
 For $\Re(s)>1$, employing the Euler product representation of $L(s,\chi_{c})$ yields
 \begin{align*}\exp\left(\frac{i}{2}(y\mathcal{L}(s,\chi_{c}))\right)
 &=\sum_{\fa\subset\ringO_{k}}\frac{\lambda_{y}(\fa)\chi_{c}(\fa)}{\norm(\fa)^s},\end{align*}  where $\lambda_{y}$ is an arithmetic function defined on the integral ideals of $k$ as 
 follows: \[\lambda_y(\fa)=\prod_{\mathfrak{p}} \lambda_y({\mathfrak{p}}^{\alpha_\mathfrak{p}}),\quad\quad\lambda_y(\mathfrak{p} ^{\alpha_\mathfrak{p}} )=G_{\alpha_\mathfrak{p}} \left(-iy\log{\norm(\mathfrak{p})}\right), \] and the function $G_{r}(u)$ is determined using the generating series 
 \begin{equation*}
\exp\left(\frac{ut}{1-t}\right)=\sum_{r=0}^{\infty}G_{r}(u)t^r,
\end{equation*}
for $u,t\in\mathbb{C}$ with $|t|<1$. The reader is referred to the proof of \cite[Lemma 4.1]{AH} for a detailed description of $\lambda_{y}$ in (Case 1) and (Case 2).

 This allows us to write $I_{y}(s,\chi_c)$ for $\Re(s)>1$ as an absolutely convergent Dirichlet series. More precisely, we get 
 \begin{align*}
 I_{y}(s,\chi_c)&=\sum_{\fa,\fb\subset\ringO_k}\frac{\lambda_{\overline{y}}(\fa)\lambda_{y}(\fb)\chi_{c}(\fa\fb^2)}{\norm(\fa)^{s}\norm(\fb)^{\overline{s}}}.
 \end{align*}

 For fixed $A>0$ and sufficiently small $\epsilon>0$, let $R_{Y, \epsilon, A}$ be the rectangle with  the vertices $1+i(\log{Y})^{2A}$, $(1+\epsilon)/2+i(\log{Y})^{2A}$, $(1+\epsilon)/2-i(\log{Y})^{2A}$, and  $1-i(\log{Y})^{2A}$. We say that an element $c\in \mathcal{C}$ is in $\mathcal{Z}^{\mathrm{c}}$, if $L(s, \chi_c)$ does not have a zero in $R_{Y, \epsilon, A}$. Otherwise, it is in $\mathcal{Z}$. Thus, $\mathcal{C}=\mathcal{Z}\cup \mathcal{Z}^{\mathrm{c}}$. Note that $\mathcal{Z}$ and $\mathcal{Z}^{\mathrm{c}}$ depend on $Y$, $\epsilon$, and $A$.
 
 For $s$ and $y$ as in Proposition \ref{mainprop1}, we have
\begin{equation}
\label{main}
\sum_{c\in \mathcal{C}}^{\star}I_{y}(s,\chi_c)\exp(-\norm(c)/Y)= \sum_{c\in \mathcal{Z}^{\mathrm{c}}}^{\star}+ \sum_{c\in \mathcal{Z}}^{\star}=\sum_{c\in \mathcal{Z}^{\mathrm{c}}}+ \sum_{c\in \mathcal{Z}}^{\star}.
\end{equation}
Observe that  $\left|I_{y}(s,\chi_c)\right|=1$. Thus, by \cite[Lemma~4.3]{AH}, an application of a zero-density estimate,  we have
$$\sum_{c\in \mathcal{Z}}^{\star}I_{y}(s,\chi_c)\exp(-\norm(c)/Y) \ll \sum_{c\in \mathcal{Z}}^{\star}\exp(-\norm(c)/Y)\ll Y^\delta,$$
for some $\delta<1$.
The application of this estimate in \eqref{main} yields
\begin{equation}
\label{main1}
\sum_{c\in \mathcal{C}}^{\star}I_{y}(s,\chi_c)\exp(-\norm(c)/Y)= \sum_{c\in \mathcal{Z}^{\mathrm{c}}} I_{y}(s,\chi_c)\exp(-\norm(c)/Y)+ O(Y^\delta).
\end{equation}

For $c\in\mathcal{Z}^{\mathrm{c}}$, we express $I_{y}(s,\chi_c)$ as the sum of an infinite series with rapidly decaying terms and a certain contour integral. In fact, following the proof of \cite[Lemma~4.3]{AH}, we get
 
\begin{align}\label{main2}I_{y}(s,\chi_c)&= \sum_{\fa,\fb\subset\mathfrak{O}_{k}}\frac{\lambda_{\overline{y}}(\fa)\lambda_{y}(\fb)\chi_{c}(\fa\fb^{2})}{\norm(\fa)^{s}\norm(\fb)^{\overline{s}}}\exp\left(-\frac{\norm(\fa\fb)}{X}\right)\nonumber\\&\hspace{2em}-\frac{1}{2\pi i} \int_{L_{Y, \epsilon, A}}\exp\left(i\frac{\overline{y}}{2}\mathcal{L}(s+u,\chi_c)\right)\exp\left(i\frac{y}{2}\mathcal{L}(\overline{s}+u,\overline{\chi_c})\right)\Gamma(u) X^u du, \end{align}
where $L_{Y, \epsilon, A}=L_1+L_2+L_3+L_4+L_5$. Let $\sigma=\Re(s)$. Then here $L_1$ is the vertical half-line given by $(1-\sigma+\frac{\epsilon}{2})+it(\log{Y})^A$ for $t\geq 1$, $L_2$ is the horizontal line segment given by $t+i(\log{Y})^A$ for $-\frac{\epsilon}{2} \leq t \leq (1-\sigma+\frac{\epsilon}{2})$, $L_3$ is the vertical line-segment given by  $-\frac{\epsilon}{2}+it(\log{Y})^A$ for $-1\leq t\leq 1$, $L_4$ is the horizontal line segment given by $t-i(\log{Y})^A$ for $-\frac{\epsilon}{2} \leq t \leq (1-\sigma+\frac{\epsilon}{2})$, and $L_5$ is the vertical half-line given by $(1-\sigma+\frac{\epsilon}{2})+it(\log{Y})^A$ for $t\leq -1$.

Combining \eqref{main1} and \eqref{main2} yields
\begin{equation}
\label{main3}
\sum_{c\in \mathcal{C}}^{\star}I_y(s,\chi_c)\exp(-\norm(c)/Y)= (I)-(II)+(III)+ O(Y^\delta),
\end{equation}
where 
\begin{align*}
(I)&=\sum _{c\in \mathcal{C}} \left ( \sum_{\fa,\fb\subset\mathfrak{O}_{k}}\frac{{\lambda}_{\overline{y}}(\fa)\lambda_{y}(\fb)\chi_{c}(\fa\fb^{2})}{\norm(\fa)^{{s}}\norm(\fb)^{\overline{s}}}\exp\left(-\frac{\norm(\fa\fb)}{X}\right)\right) \exp\left(-\frac{\norm(c)}{Y}\right),
\end{align*}
\begin{align*}
(II)&= \sum _{c\in \mathcal{Z}} \left ( \sum_{\fa,\fb\subset\mathfrak{O}_{k}}\frac{{\lambda}_{\overline{y}}(\fa)\lambda_{y}(\fb)\chi_{c}(\fa\fb^{2})}{\norm(\fa)^{{s}}\norm(\fb)^{\overline{s}}}\exp\left(-\frac{\norm(\fa\fb)}{X}\right)\right) \exp\left(-\frac{\norm(c)}{Y}\right),
\end{align*}
and
\begin{align*}
(III)&= \sum _{c\in \mathcal{Z}^{\mathrm{c}}} \left ( -\frac{1}{2\pi i} \int_{L_{Y, \epsilon, A}} \exp\left(i\frac{\overline{y}}{2}\mathcal{L}(s+u,\chi_c)\right)\exp\left(i\frac{y}{2}\mathcal{L}(\overline{s}+u,\overline{\chi_c})\right)\Gamma(u) X^u du
\right)\\&\hspace{4em}\times \exp\left(-\frac{\norm(c)}{Y}\right).
\end{align*}

 The calculations in Sections 4.2, 4.3 and 4.4 in \cite{AH} apply with very minor and obvious changes to evaluate the sums  $(I)$, $(II)$ and $(III)$ respectively. It suffices to mention that $\widetilde{M}_{s}(y_1,y_2)$ comes from the contribution of cubes to $(I)$. In fact, following the calculations in \cite[Section~4.2]{AH} we get
 
 \[(I)=\frac{3\mathrm{res}_{s=1}\zeta_k(s)}{4|H_{\langle 9\rangle}|\zeta_k(2)}Y\widetilde{M}_{s}(y_1,y_2)+o(Y),\]
 where
 
\begin{align*}\widetilde{M}_{s}(y_1,y_2)&=\sum_{r_{1},r_{2}>0}\frac{\lambda_{\overline{y}}((1-\zeta_{3})^{r_1})\lambda_{y}((1-\zeta_{3})^{r_2})}{3^{r_1s+r_2\overline{s}}}\\&\hspace{2em}\times\sum_{\substack{\fa,\fb,\fm\subset\mathfrak{O}_{k}\\\gcd(\fa\fb\fm,(3))=1\\\gcd(\fa,\fb)=1}}\frac{\lambda_{\overline{y}}(\fa^3\fm)\lambda_{y}(\fb^3\fm)}{\norm(\fa)^{3s}\norm(\fb)^{3\overline{s}}\norm(\fm)^{2\sigma}\displaystyle{\prod_{\substack{\fp|\fa\fb\fm\\\fp\; \text{prime}}}\left(1+\norm(\fp)^{-1}\right)}}\\&=\exp\left(-i \Re\left(\frac{y\log 3}{3^{\overline{s}}-1}\right)\right)\\&\hspace{2em}\times\sum_{\substack{\fa,\fb,\fm\subset\mathfrak{O}_{k}\\\gcd(\fa\fb\fm,(3))=1\\\gcd(\fa,\fb)=1}}\frac{\lambda_{\overline{y}}(\fa^3\fm)\lambda_{y}(\fb^3\fm)}{\norm(\fa)^{3s}\norm(\fb)^{3\overline{s}}\norm(\fm)^{2\sigma}\displaystyle{\prod_{\substack{\fp|\fa\fb\fm\\\fp\; \text{prime}}}\left(1+\norm(\fp)^{-1}\right)}}.
\end{align*}
One can then verify that $\widetilde{M}_{s}(y_1,y_2)$ admits the product expansion given in Theorem \ref{mainthrm} (see \cite[Section~4.6]{AH} for details). Combining the estimates obtained for $(I)$, $(II)$ and $(III)$ as  done in \cite[Section~4.5]{AH} yields
 \[\sum_{c\in \mathcal{C}}^{\star}I_y(s,\chi_c)\exp(-\norm(c)/Y)=\frac{3\mathrm{res}_{s=1}\zeta_k(s)}{4|H_{\langle 9\rangle}|\zeta_k(2)}Y\widetilde{M}_s(y_1,y_2)+o(Y), \quad\quad\text{as}\;Y\to\infty.\]
 Observe that the convergence is uniform on any sphere $|y|\leq a$. This is mainly due to the estimate $\lambda_{y}(\fa)\ll_{\epsilon,a}\norm(\fa)^{\epsilon}$ (see \cite[p.~92]{I-M}) which holds for any $\epsilon>0$ and all $|y|\leq a$ where $a$ is a positive real number.

Finally, by noting that $\mathcal{N}^{*}(Y)= \frac{3\mathrm{res}_{s=1}\zeta_k(s)}{4|H_{\langle 9\rangle}|\zeta_k(2)}Y+O(Y^{\frac12+\epsilon})$ (see \cite[Lemma~3.3]{AH}), we complete the proof of Proposition \ref{mainprop1}.

\section{Proof of Proposition \ref{mainprop2}}\label{sec:prop2}
In this section we prove Proposition \ref{mainprop2} in (Case 2). The proof of the proposition in (Case 1) is similar.  
 
We will  show that 
\begin{align*}
\widetilde{M}_s(y_1, y_2)&= \exp\left(-i\Re\left(\frac{\bar{y}\log 3}{3^{s}-1}\right)\right)\\&\hspace{1em}\times\prod_{\fp\nmid\langle3\rangle} \left( \frac{1}{\norm(\fp)+1}+\frac{1}{3}\frac{\norm(\fp)}{\norm(\fp)+1}\sum_{j=0}^{2}\exp\left(-i\Re\left(\frac{\bar{y}\zeta_{3}^{j}   \log\norm(\fp)    }{\norm(\fp)^{s}-\zeta_{3}^{j}}\right)\right) \right)
\end{align*}
satisfies $\left|\widetilde{M}_s(y_1, y_2)\right|\leq \exp\left(-C|y|^{\frac{1}{\sigma}}\left(\log\left| y\right|\right)^{\frac{1}{\sigma}-1}\right)$, as $|y|\to\infty$. 

Let $\fp$ be a prime ideal such that  $\fp\nmid\langle3\rangle$ and consider the corresponding Euler factor 
\begin{equation*}
\widetilde{M}_{s,\fp}(y_1, y_2)=\frac{1}{\norm(\fp)+1}+\frac{1}{3}\frac{\norm(\fp)}{\norm(\fp)+1}\sum_{j=0}^{2}\exp\left(-i\Re\left(\frac{\bar{y}\zeta_{3}^{j}   \log\norm(\fp)    }{\norm(\fp)^{s}-\zeta_{3}^{j}}\right)\right).
\end{equation*}
 Using the Maclaurin series of the function $f(x)=\frac{x}{1-x}$ for $|x|<1$ and noting that \\ $|e^{ib}-e^{ia}|\leq|b-a|$, we get 
\begin{align}\label{eqn:Fp}
\widetilde{M}_{s,\fp}(y_1, y_2)&= G_{\fp}(y,s)+O\left(\frac{|y|\log\norm(\fp)}{\norm(\fp)^{2\sigma}}\right),
\end{align}
where \begin{align*}G_{\fp}(y,s)&=\frac{1}{\norm(\fp)+1}+\frac{\norm(\fp)}{3(\norm(\fp)+1)}\exp\left(-i\log\norm(\fp)\frac{y_1\Re(\norm(\fp)^{-it})+y_2\Im(\norm(\fp)^{-it})}{\norm(\fp)^{\sigma}}\right)\\&\hspace{1em}+\frac{\norm(\fp)}{3(\norm(\fp)+1)}\exp\left(-i\log\norm(\fp)\frac{y_1\Re(\zeta_3\norm(\fp)^{-it})+y_2\Im(\zeta_3\norm(\fp)^{-it})}{\norm(\fp)^{\sigma}}\right)\\&\hspace{1em}+\frac{\norm(\fp)}{3(\norm(\fp)+1)}\exp\left(-i\log\norm(\fp)\frac{y_1\Re(\zeta_3^2\norm(\fp)^{-it})+y_2\Im(\zeta_3^2\norm(\fp)^{-it})}{\norm(\fp)^{\sigma}}\right).\end{align*}

Before we proceed further, we need the following lemma, the proof of which will be given at the end of this section.
\begin{lemma}\label{lem:Gbound}
Let $\varepsilon>0$ be given. Set $a_{\sigma,\varepsilon}(y)=\frac{|y|\log|y|}{\sigma(\frac{\pi}{6}-\varepsilon)}$ and $b_{\sigma,\varepsilon}(y)=\frac{|y|\log|y|}{\sigma(\frac{\pi}{9({\sqrt{3}}/{2})}+\varepsilon)}$. As $|y|\to\infty$, we have $0.097<|G_{\fp}(y,s)|<0.9978$
  for all prime ideals $\fp$ with $a_{\sigma,\varepsilon}(y)\leq\norm(\fp)^{\sigma}\leq b_{\sigma,\varepsilon}(y)$.\end{lemma}

We now use Lemma \ref{lem:Gbound} to establish the proof of Proposition \ref{mainprop2} in (Case 2). It is clear that $\left|\widetilde{M}_{s,\fp}(y_1, y_2)\right|\leq1$ for all $\fp$. Using \eqref{eqn:Fp}, we get, as $|y| \rightarrow \infty$,
\begin{align}\label{eqn:M}
\left|\widetilde{M}_{s}(y_1, y_2)\right|&\leq\prod_{\substack{\fp\\ a_{\sigma,\varepsilon}(y)\leq\norm(\fp)^{\sigma}\leq b_{\sigma,\varepsilon}(y)}}\left|\widetilde{M}_{s,\fp}(y_1, y_2)\right|\nonumber\\&=\prod_{\substack{\fp\\ a_{\sigma,\varepsilon}(y)\leq\norm(\fp)^{\sigma}\leq b_{\sigma,\varepsilon}(y)}}\left|G_{\fp}(y,s)+O\left(\frac{|y|\log\norm(\fp)}{\norm(\fp)^{2\sigma}}\right)\right|\nonumber\\&=
\prod_{\substack{\fp\\ a_{\sigma,\varepsilon}(y)\leq\norm(\fp)^{\sigma}\leq b_{\sigma,\varepsilon}(y)}}\left|G_{\fp}(y,s)\right|\nonumber\\&\hspace{6em}\times\prod_{\substack{\fp\\ a_{\sigma,\varepsilon}(y)\leq\norm(\fp)^{\sigma}\leq b_{\sigma,\varepsilon}(y)}}\left(1+O\left(\frac{|y|\log\norm(\fp)}{\norm(\fp)^{2\sigma}}\right)\right),
\end{align}
where the last equality follows from the fact that $G_{\fp}(y,s)$ is bounded away from zero in the specified range for $\norm(\fp)$ as seen in Lemma \ref{lem:Gbound}. Notice that 
\begin{equation}\label{eqn:prod}\prod_{\substack{\fp\\ a_{\sigma,\varepsilon}(y)\leq\norm(\fp)^{\sigma}\leq b_{\sigma,\varepsilon}(y)}}\left(1+O_\sigma\left(\frac{|y|\log\norm(\fp)}{\norm(\fp)^{2\sigma}}\right)\right)=\exp\left(O\left(|y|^{-1+\frac{1}{\sigma}}\log|y|\right)\right),\end{equation}
as $|y|\to\infty$. Moreover,  we have \[\Pi_{\sigma}(y)=\Pi\left((b_{\sigma,\varepsilon}(y))^{\frac{1}{\sigma}}\right)-\Pi\left((a_{\sigma,\varepsilon}(y))^{\frac{1}{\sigma}}\right)\gg_{\sigma,\varepsilon}|y|^{\frac{1}{\sigma}}\left(\log\left| y\right|\right)^{\frac{1}{\sigma}-1},\quad\quad\text{as}\;y\to\infty,\]
where $\Pi(x)$ is the number of prime ideals in $\ringO_{k}$ whose norm is less than or equal to $x$. 

Using this observation along with the upper bound for $G_{\fp}(y,s)$ given in Lemma \ref{lem:Gbound}, we see that for $|y|$ large enough, we have \begin{align}\label{eqn:Gbound}\prod_{\substack{\fp\\ a_{\sigma,\varepsilon}(y)\leq\norm(\fp)^{\sigma}\leq b_{\sigma,\varepsilon}(y)}}\left|G_{\fp}(y,s)\right|&\leq\exp\left(\Pi_{\sigma}(y)\log(0.9978)\right)\nonumber\\&\leq\exp\left(-C|y|^{\frac{1}{\sigma}}\left(\log\left| y\right|\right)^{\frac{1}{\sigma}-1}\right),\end{align} where $C$ is a positive constant that depends only on $\sigma$ and $\varepsilon$.
Therefore, applying the estimates \eqref{eqn:prod} and \eqref{eqn:Gbound} in \eqref{eqn:M} yields
\[\left|\widetilde{M}_{s}(y_1,y_2)\right|\leq\exp(-C|y|^{\frac{1}{\sigma}}\left(\log\left| y\right|\right)^{\frac{1}{\sigma}-1}),\] as desired.

Finally, we prove Lemma \ref{lem:Gbound}.
\begin{proof}[Proof of Lemma \ref{lem:Gbound}] 
We write
\begin{align*}
G_{\fp}(y,s)&=\frac{1}{\norm(\fp)+1}\\&\hspace{2em}+\frac{\norm(\fp)}{3(\norm(\fp)+1)}\exp\left(-i\log\norm(\fp)\frac{y_1\Re(\norm(\fp)^{-it})+y_2\Im(\norm(\fp)^{-it})}{\norm(\fp)^{\sigma}}\right)H_{\fp}(y,s),
\end{align*} where
\begin{align*}H_{\fp}(y,s)&=1+\exp\left(-i\log\norm(\fp)\frac{y_1\Re((\zeta_3-1)\norm(\fp)^{-it})+y_2\Im((\zeta_3-1)\norm(\fp)^{-it})}{\norm(\fp)^{\sigma}}\right)\\&\hspace{2em}+\exp\left(-i\log\norm(\fp)\frac{y_1\Re((\zeta_3^2-1)\norm(\fp)^{-it})+y_2\Im((\zeta_3^2-1)\norm(\fp)^{-it})}{\norm(\fp)^{\sigma}}\right).\end{align*}
Observe that $\zeta_3-1=\sqrt{3}e^{-i\frac{\pi}{6}}$ and $\zeta_3^2-1=\sqrt{3}e^{i\frac{\pi}{6}}$. Let $y_1=|y|\Re(e^{i\theta_y})$ and $y_2=|y|\Im(e^{i\theta_y})$ for some $\theta_y\in[0,2\pi]$. It follows that

\begin{align*}
H_{\fp}(y,s)
&=1+\exp\left(-i\log\norm(\fp)\frac{\sqrt{3}|y|}{\norm(\fp)^{\sigma}}\cos(\theta_y+\tfrac{\pi}{6}+t\log\norm(\fp))\right)\\&\hspace{2em}+\exp\left(-i\log\norm(\fp)\frac{\sqrt{3}|y|}{\norm(\fp)^{\sigma}}\cos(\theta_y-\tfrac{\pi}{6}+t\log\norm(\fp))\right).\end{align*}
Using the identity 
\[|1+e^{iax}+e^{ibx}|=\sqrt{3+4\cos(\frac{a+b}{2}x)\cos(\frac{a-b}{2}x)+2\cos((a-b)x)},\] we have 
\begin{align*}\left|G_{\fp}(y,s)\right|&\leq\frac{1}{\norm(\fp)+1}+\frac{\norm(\fp)}{3(\norm(\fp)+1)}|H_{\fp}(y,s)|,\end{align*} where we can write \begin{align*}|H_{\fp}(y,s)|&=\sqrt{3+2J_{\fp}(y,s)},\end{align*}
with \begin{align*}J_{\fp}(y,s)&=2\cos\left(\tfrac{3|y|\log\norm(\fp)}{2\norm(\fp)^{\sigma}}\cos(\theta_y+t\log\norm(\fp))\right)\cos\left(\tfrac{\sqrt{3}|y|\log\norm(\fp)}{2\norm(\fp)^{\sigma}}\sin(\theta_y+t\log\norm(\fp))\right)\\&\hspace{2em}+\cos\left(\tfrac{\sqrt{3}|y|\log\norm(\fp)}{\norm(\fp)^{\sigma}}\sin(\theta_y+t\log\norm(\fp))\right).\end{align*}
For $|y|$ large enough, we are interested in the prime ideals $\fp$ for which 
\[H_{\fp}(y,s)<3.\]
To this end, we show that for a positive proportion of primes we have 
\begin{equation}\label{eqn:bound<1}\cos\left(\tfrac{3|y|\log\norm(\fp)}{2\norm(\fp)^{\sigma}}\cos(\theta_y+t\log\norm(\fp))\right)\cos\left(\tfrac{\sqrt{3}|y|\log\norm(\fp)}{2\norm(\fp)^{\sigma}}\sin(\theta_y+t\log\norm(\fp))\right)<1.\end{equation}
We  must also show that for those prime ideals, we have $|G_{\fp}(y,s)|$ is bounded away from zero.\\

Consider the following sets 
\[A=\bigcup_{N=0}^{\infty}\left(\left(2N\pi-\frac{\pi}{6},2N\pi+\frac{\pi}{6}\right)\cup\left(2N\pi+\frac{5\pi}{6},2N\pi+\frac{7\pi}{6}\right)\right)\]
and 
\[B=\bigcup_{N=0}^{\infty}\left(\left[2N\pi+\frac{\pi}{6},2N\pi+\frac{5\pi}{6}\right]\cup\left[2N\pi+\frac{7\pi}{6},2N\pi+\frac{11\pi}{6}\right]\right).\]
Now let $\fp$ be such that 
\begin{equation}\label{eqn:zp-condition}
\frac{\pi}{9(\sqrt{3}/2)}<\frac{|y|\log\norm(\fp)}{\norm(\fp)^{\sigma}}<\frac{\pi}{6}.
\end{equation}
Given a prime ideal $\fp$ that satisfies \eqref{eqn:zp-condition}, we consider two cases. \\

\noindent\underline{{\bf 1. $\theta_y+t\log(\norm(\fp))\in A$:}}\\

 Assume that $\theta_y+t\log(\norm(\fp))\in \left(2N\pi-\frac{\pi}{6},2N\pi+\frac{\pi}{6}\right)$ for some non-negative integer $N$. It follows that \[\frac{\pi}{6}<\tfrac{3|y|\log\norm(\fp)}{2\norm(\fp)^{\sigma}}\cos(\theta_y+t\log\norm(\fp))<\frac{\pi}{4}.\] Hence, 
\[\frac{\sqrt{2}}{2}<\cos\left(\tfrac{3|y|\log\norm(\fp)}{2\norm(\fp)^{\sigma}}\cos(\theta_y+t\log\norm(\fp))\right)<\frac{\sqrt{3}}{2},\] and so \eqref{eqn:bound<1} holds. In particular, we get
\begin{equation}\label{eq:upperbound1}
|G_{\fp}(y,s)|< 0.9799.
\end{equation}
Moreover, \[0<\frac{\sqrt{3}|y|\log\norm(\fp)}{2\norm(\fp)^{\sigma}}\sin(\theta_y+t\log\norm(\fp))< \frac{\pi\sqrt{3}}{24} ,\] and so 
\[\cos\left(\tfrac{\sqrt{3}|y|\log\norm(\fp)}{2\norm(\fp)^{\sigma}}\sin(\theta_y+t\log\norm(\fp))\right)>\cos(\tfrac{\pi\sqrt{3}}{24}). \]
It follows that $J_{\fp}(y,s)>\sqrt{2}\cos(\tfrac{\sqrt{3}\pi}{24})-1=0.378...$. Hence,
\begin{align}\label{eqn:bounded-away1}|G_{\fp}(y,s)|&\geq \frac{\norm(\fp)}{3(\norm(\fp)+1)}\sqrt{3+2J_{\fp}(y,s)}-\frac{1}{\norm(\fp)+1}\nonumber\\&>\frac{\norm(\fp)}{3(\norm(\fp)+1)}\sqrt{3+2(0.378)}-\frac{1}{\norm(\fp)+1}\nonumber\\&>0.097.\end{align}
 A similar result holds if $\theta_y+t\log(\norm(\fp))\in \left(2N\pi+\frac{5\pi}{6},2N\pi+\frac{7\pi}{6}\right)$ for some non-negative integer $N$ since the cosine function is even.\\

\noindent \underline{ {\bf 2. $\theta_y+t\log(\norm(\fp))\in B$:}}\\

 Assume that $\theta_y+t\log(\norm(\fp))\in \left[2N\pi+\frac{\pi}{6},2N\pi+\frac{5\pi}{6}\right]$ for some non-negative integer $N$. It follows that \[\frac{\pi}{18}<\tfrac{\sqrt{3}|y|\log\norm(\fp)}{2\norm(\fp)^{\sigma}}\sin(\theta_y+t\log\norm(\fp))<\frac{\sqrt{3}\pi}{12}.\] 
Hence, 
\[\cos\left(\frac{\sqrt{3}\pi}{12}\right)<\cos\left(\tfrac{\sqrt{3}|y|\log\norm(\fp)}{2\norm(\fp)^{\sigma}}\sin(\theta_y+t\log\norm(\fp))\right)<\cos\left(\frac{\pi}{18}\right),\] and so \eqref{eqn:bound<1} holds. In particular, we get
\begin{equation}\label{eq:upperbound2}
|G_{\fp}(y,s)|< 0.9978.
\end{equation}
Moreover, \[\left|\frac{3|y|\log\norm(\fp)}{2\norm(\fp)^{\sigma}}\cos(\theta_y+t\log\norm(\fp))\right|< \frac{\sqrt{3}\pi}{8} ,\] and so 
\[\cos\left(\tfrac{3|y|\log\norm(\fp)}{2\norm(\fp)^{\sigma}}\cos(\theta_y+t\log\norm(\fp))\right)>\cos(\tfrac{\sqrt{3}\pi}{8}). \]
It follows that $J_{\fp}(y,s)>2\cos(\tfrac{\sqrt{3}\pi}{8})\cos(\tfrac{\sqrt{3}\pi}{12})-1=0.3977...$. Hence,
\begin{align}\label{eqn:bounded-away2}|G_{\fp}(y,s)|\geq\frac{\norm(\fp)}{3(\norm(\fp)+1)}\sqrt{3+2(0.3977)}-\frac{1}{\norm(\fp)+1}>0.099.\end{align}
If $\theta_y+t\log(\norm(\fp))\in \left[2N\pi+\frac{7\pi}{6},2N\pi+\frac{11\pi}{6}\right]$ for some non-negative integer $N$, a similar result holds since the cosine function is even. \\  

To finish the proof, observe that as $|y|\to\infty$, we have
\begin{equation*}
\frac{|y|\log|y|}{\sigma(\frac{\pi}{6}-\varepsilon)}<\norm(\fp)^{\sigma}<\frac{|y|\log|y|}{\sigma(\frac{\pi}{9(\sqrt{3}/{2})}+\varepsilon)}\quad\implies\quad \frac{\pi}{9(\sqrt{3}/2)}<\frac{|y|\log\norm(\fp)}{\norm(\fp)^{\sigma}}<\frac{\pi}{6},
\end{equation*}
for any $\varepsilon>0$. Setting $a_{\sigma,\varepsilon}(y)=\frac{|y|\log|y|}{\sigma(\frac{\pi}{6}-\varepsilon)}$ and $b_{\sigma,\varepsilon}(y)=\frac{|y|\log|y|}{\sigma(\frac{\pi}{9({\sqrt{3}}/{2})}+\varepsilon)}$, we deduce that for all prime ideals $\fp$ with $a_{\sigma,\varepsilon}(y)\leq\norm(\fp)^{\sigma}\leq b_{\sigma,\varepsilon}(y)$, we have 
$$0.097<|G_{\fp}(y,s)|<0.9978,$$ as $|y|\to\infty$.
\end{proof}

\subsection*{Acknowledgement} The authors would like to thank the referee for the valuable comments and suggestions.

\end{document}